\documentclass[12pt,twoside]{amsart}
\usepackage{amsmath, amsthm, amscd, amsfonts, amssymb, graphicx}
\newtheorem{thm}{Theorem}[section]

\newtheorem{corollary}[thm]{Corollary}
\newtheorem{lemma}[thm]{Lemma}

\theoremstyle{definition}
\newtheorem{defn}{Definition}[section]

\theoremstyle{remark}
\newtheorem{rem}{Remark}[section]
\newtheorem{example}{Example}[section]

\numberwithin{equation}{section}
\begin{document}
\begin{center}\large{{\bf{Application of Geometric Calculus in Numerical Analysis and Difference Sequence Spaces }}} 
\vspace{0.5cm}

Khirod Boruah and Bipan Hazarika$^{\ast}$ 

\vspace{0.5cm}
Department of Mathematics, Rajiv Gandhi University, Rono Hills, Doimukh-791112, Arunachal Pradesh, India\\

Email: khirodb10@gmail.com; bh\_rgu@yahoo.co.in
\thanks{$^{\ast}$ The corresponding author.}
\end{center}
\title{}
\author{}
\thanks{{\today}}
\begin{abstract} The main purpose of this paper is to introduce the geometric difference sequence space
 $l_\infty^{G} (\Delta_G)$  and prove that $l_\infty^{G} ({\Delta}_{G})$ is a Banach space with respect to the norm $\left\|.\right\|^G_{{\Delta}_G}.$ Also we compute the $\alpha$-dual, $\beta$-dual and $\gamma$-dual spaces. Finally we obtain the Geometric Newton-Gregory interpolation formulae.\\
\parindent=5mm
\noindent{\footnotesize {\bf{Keywords and phrases:}}} Difference sequence spaces; dual space; Geometric Calculus; interpolation formula.\\
{\footnotesize {\bf{AMS subject classification \textrm{(2000)}:}}} 26A06, 11U10, 08A05, 46A45.
\end{abstract}
\maketitle

\maketitle

\pagestyle{myheadings}
\markboth{\rightline {\scriptsize  Khirod Boruah, Bipan Hazarika}}
        {\leftline{\scriptsize  Geometric difference sequence spaces over ...}}

\maketitle
\section{Introduction and Notations}
In 1967 Robert Katz and Michael Grossman created the first system of non-Newtonian calculus, which we call the geometric calculus. In 1970 they had created an infinite family of non-Newtonian calculi, each of which differs markedly from the classical calculus of Newton and Leibniz. Among other things, each non-Newtonian calculus possesses four operators : a gradient (i.e. an average rate of change), a derivative, an average, and an integral. For each non-Newtonian calculus  there is a characteristic class of functions having a constant derivative.

In view of pioneering work carried out in this area by Grossman and Katz \cite{GrossmanKatz}  we will call this calculus as multiplicative calculus, although the term of exponential calculus can also
be used. The operations of multiplicative calculus will be called as multiplicative derivative and multiplicative integral. We refer to Grossman and Katz \cite{GrossmanKatz}, Stanley \cite{Stanley}, Bashirov et
al. \cite{BashirovMisirh,BashirovKurpinar}, Grossman  \cite{Grossman83} for elements of multiplicative calculus and
its applications. An extension of multiplicative calculus to functions of complex variables is
handled in Bashirov and R\i za \cite{BashirovRiza}, Uzer \cite{Uzer10}, Bashirov et al. \cite{BashirovKurpinar}, \c{C}akmak and Ba\c{s}ar \cite{CakmakBasar}, Tekin and Ba\c{s}ar\cite{TekinBasar}, T\"{u}rkmen and Ba\c{s}ar \cite{TurkmenBasar}. In \cite{KADAK3},  Kadak and  \"{O}zl\"{u}k studied the generalized Runge-Kutta method with
respect to non-Newtonian calculus.  Kadak et al \cite{KadakEfe,kadak2} studied certain new types of sequence spaces over the Non-Newtonian Complex Field.

Geometric calculus is an alternative to the usual calculus of Newton and Leibniz. It provides differentiation and integration tools based on multiplication instead of addition. Every property in Newtonian calculus has an analog in multiplicative calculus. Generally speaking multiplicative calculus is a methodology that allows one to have a different look at problems which can be investigated via calculus. In some cases, for example for growth related problems, the use of multiplicative calculus is advocated instead of a traditional Newtonian one.
 
The main aim of this paper is to construct the difference sequence space $l_\infty^{G} \left({\Delta}_G\right)$ over geometric  complex numbers which forms a Banach space with the norm defined on it and obtain the Geometric Newton-Gregory interpolation formulae which are more useful than Newton-Gregory interpolation formulae.

We should know that all concepts in classical arithmetic have natural counterparts in $\alpha-arithmetic.$ Consider any generator $\alpha$ with range $A\subseteq \mathbb{C}.$ By $\alpha- arithmetic,$ we mean the arithmetic whose domain is $A$ and operations are defined as follows. 
For $x, y \in A$ and any generator $\alpha,$
\begin{align*}
&\alpha -addition &x\dot{+}y &=\alpha[\alpha^{-1}(x) + \alpha^{-1}(y)]\\
&\alpha-subtraction &x\dot{-}y&=\alpha[\alpha^{-1}(x) - \alpha^{-1}(y)]\\
&\alpha-multiplication &x\dot{\times}y &=\alpha[\alpha^{-1}(x) \times \alpha^{-1}(y)]\\
&\alpha-division &\dot{x/y}&=\alpha[\alpha^{-1}(x) / \alpha^{-1}(y)]\\
&\alpha-order &x\dot{<}y &\Leftrightarrow \alpha^{-1}(x) < \alpha^{-1}(y).
\end{align*}
If we choose \textit{$exp$} as an $\alpha-generator$ defined by $\alpha (z)= e^z$ for $z\in \mathbb{C}$ then $\alpha^{-1}(z)=\ln z$ and $\alpha-arithmetic$ turns out to Geometric arithmetic.
\begin{align*}
&\alpha -addition &x\oplus y &=\alpha[\alpha^{-1}(x) + \alpha^{-1}(y)]& = e^{(\ln x+\ln y)}& =x.y ~geometric ~addition\\
&\alpha-subtraction &x\ominus y&=\alpha[\alpha^{-1}(x) - \alpha^{-1}(y)]&= e^{(\ln x-\ln y)} &=  x\div y, y\ne 0 ~geometric ~subtraction\\
&\alpha-multiplication &x\odot y &=\alpha[\alpha^{-1}(x) \times\alpha^{-1}(y)]& = e^{(\ln x\times\ln y)} & = ~x^{\ln y}~ geometric ~multiplication\\
&\alpha-division &x\oslash y&=\alpha[\alpha^{-1}(x) / \alpha^{-1}(y)] & = e^{(\ln x\div \ln y)}& = x^{\frac{1}{\ln y}}, y\ne 1 ~ geometric ~division.
\end{align*}
In \cite{TurkmenBasar} defined the geometric complex numbers $\mathbb{C}(G)$ as follows:
\[\mathbb{C}(G):=\{ e^{z}: z\in \mathbb{C}\} = \mathbb{C}\backslash \{0\}.\]
Then $(\mathbb{C}(G), \oplus, \odot)$ is a field with geometric zero $1$ and geometric identity $e.$\\
Then for all $x, y\in \mathbb{C}(G)$
\begin{itemize}
\item{ $x\oplus y=xy$}
\item{ $x\ominus y=x/y$}
\item{ $x\odot y=x^{\ln y}=y^{\ln x}$}
\item{ $x\oslash y$ or $\frac{x}{y}G=x^{\frac{1}{\ln y}}, y\neq 1$}
\item{ $x^{2_G}= x \odot x=x^{\ln x}$}
\item{ $x^{p_G}=x^{\ln^{p-1}x}$}
\item{ ${\sqrt{x}}^G=e^{(\ln x)^\frac{1}{2}}$}
\item{ $x^{-1_G}=e^{\frac{1}{\log x}}$}
\item{ $x\odot e=x$ and $x\oplus 1= x$}
\item{ $e^n\odot x=x^n=x\oplus x\oplus .....(\text{upto $n$ number of $x$})$}
\item{
\begin{equation*}
\left|x\right|^G=
\begin{cases}
x, &\text{if $x>1$}\\
1,&\text{if $x=1$}\\
\frac{1}{x},&\text{if $x<1$}
\end{cases}
\end{equation*}}
Thus $\left|x\right|^G\geq 1.$
\item{ ${\sqrt{x^{2_G}}}^G=\left|x\right|^G$}
\item{ $\left|e^y\right|^G=e^{\left|y\right|}$}
\item{ $\left|x\odot y\right|^G=\left|x\right|^G \odot \left|y\right|^G$}
\item{ $\left|x\oplus y\right|^G \leq\left|x\right|^G \oplus \left|y\right|^G$}
\item{ $\left|x\oslash y\right|^G=\left|x\right|^G \oslash \left|y\right|^G$}
\item{ $\left|x\ominus y\right|^G\geq\left|x\right|^G \ominus \left|y\right|^G$}
\item{ $0_G \ominus 1_G\odot\left(x \ominus y\right)=y\ominus x\,, i.e.$ in short $\ominus \left(x \ominus y\right)= y\ominus x.$}
\end{itemize}
 Let $l_{\infty},c$ and $c_0$ be the linear spaces of complex bounded, convergent and null  sequences, respectively, normed by
\[||x||_\infty=\sup_k|x_k|.\]
 T\"{u}rkmen and Ba\c{s}ar \cite{TurkmenBasar}  have proved that
\[\omega(G)=\{(x_k): x_k \in \mathbb{C}(G)\, \text{for all}\, k\in \mathbb{N}\}\]
is a vector space over $\mathbb{C}(G)$ with respect to the algebraic operations $\oplus$ addition and $\odot$ multiplication 
\begin{align*}
\oplus : \omega(G) \times \omega (G) &\rightarrow \omega (G)\\
                          (x, y)&\rightarrow x \oplus y =(x_k) \oplus (y_k)=(x_ky_k)\\
\odot : \mathbb{C(G)} \times \omega (G) &\rightarrow \omega (G)\\
													(\alpha, y)&\rightarrow \alpha \odot y=\alpha \odot (y_k)=(\alpha^{\ln y_k}),
\end{align*}
where $x=(x_k), y=(y_k) \in \omega (G)$ and $\alpha \in \mathbb{C}(G).$ Then
\begin{align*}
l_\infty(G) &=\{x=(x_k) \in \omega (G): \sup_{k\in \mathbb{N}}|x_k|^G< \infty\}\\
c(G)        &= \{x=(x_k) \in \omega (G): {_G\lim_{k\rightarrow \infty}}|x_k\ominus l|^G=1\}\\
c_0(G)      &= \{x=(x_k) \in \omega (G): {_G\lim_{k\rightarrow \infty}} x_k=1\}, \text{where $_G\lim$ is the geometric limit}\\
l_p(G)      &= \{x=(x_k) \in \omega (G):{_G\sum^\infty_{k=0}}\left(|x_k|^G\right)^{p_G} <\infty\}, \text{~where ${_G\sum}$ is the geometric sum},
\end{align*}
are classical sequence spaces over the field $\mathbb{C}(G).$ Also it is shown that $l_{\infty}(G),$ $c(G)$ and $c_0(G)$ are Banach spaces with the norm 
\[||x||^{G}=\sup_{k}|x_k|^{G}, x=(x_1,x_2,x_3,...)\in \lambda(G), \lambda\in \{l_{\infty},c, c_0\}.\]
For the convenience, in this paper we denote $l_\infty(G), c(G), c_0(G),$ respectively as $l_\infty^G, c^G, c_0^G.$
\section{New geometric sequence space}
In 1981, Kizmaz \cite{Kizmaz} introduced the notion of difference sequence spaces using forward difference operator $\Delta$ and studied the classical difference sequence spaces $\ell _{\infty }(\Delta ),$ $c(\Delta
),$ $c_{0}(\Delta ).$ In this section we define the following new geometric sequence space 
\[l_\infty^G(\Delta_G)= \{x=(x_k) \in \omega (G): \Delta_G x\in l_\infty^G\},
\text{~where~} {\Delta}_G x=x_k \ominus x_{k+1}.\]
\begin{thm}\label{eight} The space  $l_\infty^{G} \left({\Delta}_G\right)$
is a normed linear space w.r.t. the norm
\begin{equation*}
\left\|x\right\|^G_{{\Delta}_G}=\left|x_1\right|^G\oplus\left\|{\Delta}_Gx\right\|^G_\infty.
\end{equation*}
\end{thm} 
\begin{proof} For $x=(x_k), y=(y_k) \in l_\infty^{G} \left({\Delta}_G\right),$
\begin{align*}
N1.\quad \left\|x\right\|^G_{{\Delta}^G} &=\left|x_1\right|^G\oplus\left\|{\Delta}^Gx\right\|^G_\infty\\
&=\left|x_1\right|^G.\sup_k\left|x_k\ominus x_{k+1}\right|^G\\
&\geq 1, \quad \text{since $\left|x_1\right|^G\geq 1$ and $\left|x_k\ominus x_{k+1}\right|^G\geq 1.$}
\end{align*}
\begin{align*}
N2. \quad \left\|x\right\|^G_{{\Delta}_G} =1 &\Leftrightarrow \left|x_1\right|^G\oplus\left\|{\Delta}_Gx\right\|^G_\infty=1\\
&\Leftrightarrow \left|x_1\right|^G.\sup_k\left|x_k\ominus x_{k+1}\right|^G=1 ~ \forall k\\
&\Leftrightarrow \left|x_1\right|^G=1  \text{~and $\left|x_k\ominus x_{k+1}\right|^G= 1$}\\
&\Leftrightarrow x_1=1  \text{~and $x_k\ominus x_{k+1}=1 ~ \forall k$}\\
&\Leftrightarrow x_1=1 \text{~and $x_k\slash x_{k+1}=1 ~ \forall k$}\\
&\Leftrightarrow x_1=1  \text{~and $x_k= x_{k+1} ~~ \forall k$}\\
&\Leftrightarrow x_k=1 ~\forall k\\
 &\Leftrightarrow x=(1,1,1,1,.........)=0_G.
\end{align*}
\begin{align*}
N3. \quad  \left\|x\oplus y\right\|^G_{{\Delta}_G}&=\left|x_1\oplus y_1\right|^G\oplus \left\|{\Delta}_G(x_k\oplus y_k)\right\|^G_\infty\\
&=\left|x_1\oplus y_1\right|^G\oplus \left\|{\Delta}_G(x_ky_k)\right\|^G_\infty \\
&=\left|x_1\oplus y_1\right|^G\oplus \sup_k\left|x_ky_k\ominus x_{k+1}y_{k+1}\right|^G\\
&=\left|x_1\oplus y_1\right|^G\oplus \sup_k\left|\frac{x_ky_k}{ x_{k+1}y_{k+1}}\right|^G\\
&=\left|x_1\oplus y_1\right|^G\oplus \sup_k\left|\frac{x_k}{x_{k+1}}.\frac{ y_k}{y_{k+1}}\right|^G\\
&=\left|x_1\oplus y_1\right|^G\oplus \sup_k\left|\frac{x_k}{x_{k+1}}\oplus\frac{ y_k}{y_{k+1}}\right|^G\\
&\leq\left|x_1\oplus y_1\right|^G\oplus \sup_k\left\{\left|\frac{x_k}{x_{k+1}}\right|^G\oplus\left|\frac{ y_k}{y_{k+1}}\right|^G\right\}\\
&=\left|x_1\oplus y_1\right|^G\oplus \sup_k\left\{\left|x_k\ominus x_{k+1}\right|^G\oplus\left| y_k\ominus{y_{k+1}}\right|^G\right\}\\
&=\left|x_1\oplus y_1\right|^G\oplus \sup_k\left\{\left|{{\Delta}_G}x\right|^G\oplus\left|{{\Delta}_G}y\right|^G\right\}\\
&\leq\left|x_1\right|^G\oplus \left|y_1\right|^G\oplus \sup_k\left\{\left|{{\Delta}_G}x\right|^G\right\}\oplus \sup_k\left\{\left|{{\Delta}_G}y\right|^G\right\}\\
&=\left[\left|x_1\right|^G\oplus \sup_k\left\{\left|{{\Delta}_G}x\right|^G\right\}\right]\oplus \left[\left|y_1\right|^G\oplus \sup_k\left\{\left|{{\Delta}_G}y\right|^G\right\}\right]\\
&= \left\|x\right\|^G_{{\Delta}_G} \oplus \left\|y\right\|^G_{{\Delta}_G}.				
\end{align*}							
\begin{align*}
N4.\quad						
      \left\|\alpha\odot x\right\|^G_{{\Delta}^G} &=\left|\alpha\odot x_1\right|^G\oplus\left\|{\Delta}_G(\alpha\odot x)\right\|^G_\infty, \quad \alpha\in \mathbb{C}(G)\\
																			                                            &=\left|\alpha\right|\odot \left|x_1\right|^G\oplus\left\|\alpha\odot x_k\ominus \alpha\odot x_{k+1}\right\|^G_\infty\\
																			                                            &=\left|\alpha\right|\odot \left|x_1\right|^G\oplus\left\|\alpha\odot (x_k\ominus x_{k+1})\right\|^G_\infty\\
																																									&=\left|\alpha\right|\odot \left|x_1\right|^G\oplus\left|\alpha\right|\odot\left\|x_k\ominus x_{k+1}\right\|^G_\infty\\
																																									&=\left|\alpha\right|\odot \left[\left|x_1\right|^G\oplus\left\|{\Delta}_G x\right\|^G_\infty\right]\\
																																									&=\left|\alpha\right|\odot\left\|x\right\|^G_{{\Delta}_G}.
\end{align*}
Thus $\left\|.\right\|^G_{{\Delta}_G}$ is a norm on $\mathbb{C}(G).$
\end{proof} 
\begin{thm} The space $l_\infty^{G} \left({\Delta}_G\right)$
is a Banach space w.r.t. the norm $\left\|.\right\|^G_{{\Delta}_G}.$
\end{thm}
\begin{proof}
Let $(x_n)$ be a Cauchy sequence in $l_\infty^{G} \left({\Delta}_G\right),$ where  
$x_n= \left(x_k^{(n)}\right)=\left(x_1^{(n)}, x_2^{(n)}, x_3^{(n)},........\right)$ $\forall n \in \mathbb{N},x_k^{(n)}$ is the $k^{th}$ coordinate of $x_n.$ Then
\begin{align*}
\left\|x_n\ominus x_m\right\|^G_{{\Delta}_G}&=\left|x_1^{(n)}\ominus x_1^{(m)}\right|^G \oplus \left\|{\Delta}_G x_n\ominus {\Delta}_Gx_m\right\|^G_\infty \rightarrow 1 \text{~as $m, n\rightarrow\infty$}\\
&=\left|x_1^{(n)}\ominus x_1^{(m)}\right|^G \oplus \left\|(x_k^{(n)}\ominus x_{k+1}^{(n)})\ominus (x_k^{(m)}\ominus x_{k+1}^{(m)})\right\|^G_\infty\rightarrow \, 1\\
&=\left|x_1^{(n)}\ominus x_1^{(m)}\right|^G \oplus \left\|(x_k^{(n)}\ominus x_k^{(m)})\ominus (x_{k+1}^{(n)}\ominus x_{k+1}^{(m)})\right\|^G_\infty\rightarrow \, 1\\
&=\left|x_1^{(n)}\ominus x_1^{(m)}\right|^G \oplus \sup_k\left|(x_k^{(n)}\ominus x_k^{(m)})\ominus (x_{k+1}^{(n)}\ominus x_{k+1}^{(m)})\right|^G\rightarrow 1\text{~as $m, n\rightarrow \infty$}.
\end{align*}
This implies that $\left|x_k^{(n)}\ominus x_k^{(m)}\right|^G\rightarrow 1\mbox{~as~} n, m\rightarrow\infty~~ \forall~ k\in\mathbb{N},$ \text{~since $\left|x_k^{(n)}\ominus x_k^{(m)}\right|^G\geq 1.$}\\
Therefore for fixed $k,$ $k^{\text{th}}$ co-ordinates of all sequences form a Cauchy sequence in $\mathbb{C}(G)$\\
i.e. $x^{(n)}_k=(x^{(1)}_k, x^{(2)}_k, x^{(3)}_k,x^{(4)}_k,.........)$ is a Cauchy sequence. Then by the completeness of $\mathbb{C}(G), (x^{(n)}_k)$ converges to $x_k$ (say) as follows:
\[\begin{matrix}
x_1   &=(&x^{(1)}_1,&x^{(1)}_2,&x^{(1)}_3,&\cdots,&x^{(1)}_k,&\cdots)\\
x_2   &=(&x^{(2)}_1,&x^{(2)}_2,&x^{(2)}_3,&\cdots,&x^{(2)}_k,&\cdots)\\
x_3   &=(&x^{(3)}_1,&x^{(3)}_2,&x^{(3)}_3,&\cdots,&x^{(3)}_k,&\cdots)\\
\vdots&  &\vdots    &\vdots    &\vdots    &       &\vdots    & \\
x_m   &=(&x^{(m)}_1,&x^{(m)}_2,&x^{(m)}_3,&\cdots,&x^{(m)}_k,&\cdots)\\
\vdots&  &\vdots    &\vdots    &\vdots    &       &\vdots    & \\
x_n   &=(&x^{(n)}_1,&x^{(n)}_2,&x^{(n)}_3,&\cdots,&x^{(n)}_k,&\cdots)\\
\vdots&  &\vdots    &\vdots    &\vdots    &       &\vdots    & \\
\downarrow& &\downarrow&\downarrow&\downarrow& &\downarrow& \\
x     &=(&x_1,&x_2,&x_3,&\cdots,&x_k,&\cdots)
\end{matrix}\]
~i.e. \[{_G\lim_{n \to \infty}}x^{(n)}_k=x_k~ \forall k\in \mathbb{N}.\]
Further for each $\varepsilon> 1, \exists N=N(\varepsilon)$ s.t. $\forall \, n, m\geq N$ we have\\
\[|x^{(n)}_1 \ominus x^{(m)}_1|^G<\varepsilon,|x^{(n)}_{k+1} \ominus x^{(m)}_{k+1}\ominus (x^{(n)}_k \ominus x^{(m)}_k)|^G<\varepsilon \]
and \[{_G\lim_{m \to \infty}}|x^{(n)}_1 \ominus x^{(m)}_1|^G =|x^{(n)}_1 \ominus x_1|^G< \varepsilon.\]
This implies
 \[_G\lim_{m \to \infty}|(x^{(n)}_{k+1}\ominus x^{(m)}_{k+1})\ominus (x^{(n)}_k\ominus x^{(m)}_k)|^G= |(x^{(n)}_{k+1} \ominus x_{k+1})\ominus (x^{(n)}_k\ominus x_k)|^G <\varepsilon~ \forall~ n\geq N.\]
Since $\varepsilon$ is independent of $k,$
\begin{align*}
  &\sup_k|(x^{(n)}_{k+1} \ominus x_{k+1})\ominus (x^{(n)}_k\ominus x_k)|^G<\varepsilon.\\
\Rightarrow &\sup_k|(x^{(n)}_{k+1} \ominus x^{(n)}_k)\ominus (x_{k+1}\ominus x_k)|^G= \left\|\Delta_G x_n\ominus \Delta_G x\right\|^G_\infty <\varepsilon.
\end{align*}
Consequently we have $\left\|x_n\ominus x\right\|^G_{\Delta_G}=|x^{(n)}_1 \ominus x_1|^G \oplus \left\|\Delta_G x_n\ominus \Delta_G x\right\|^G_\infty < {\varepsilon}^2~ \forall~ n\geq N.$\\
Hence we obtain $x_n\rightarrow x$ as $n\rightarrow \infty.$ \\
Now we must show that $x\in l^G_\infty(\Delta_G).$
We have 
\begin{align*}
|x_k\ominus x_{k+1}|^G&=|x_k\ominus x^N_k\oplus x^N_k\ominus x^N_{k+1}\oplus x^N_{k+1}\ominus x_{k+1}|^G\\
                  &\leq |x^N_k\ominus x^N_{k+1}|^G\oplus ||x^N\ominus x||^G_{\Delta_G}= O(e).
\end{align*}
This implies $x=(x_k)\in l^G_\infty(\Delta_G).$
\end{proof}

Furthermore since $l^G_\infty(\Delta_G)$ is a Banach space with continuous coordinates (that is  $\left\|x_n\ominus x\right\|^\infty_{\Delta_G}\rightarrow 1$ implies $|x^{(n)}_k \ominus x_k|^G\rightarrow 1$ for each $k\in \mathbb{N},$ as $n\rightarrow \infty )$ it is a BK-space.

\begin{rem}
The spaces \begin{enumerate}
\item[(a)] $c^{G}(\Delta_{G})=\{(x_k)\in w(G): \Delta_{G}x_k\in c^{G}\}$
\item[(b)] $c_{0}^{G}(\Delta_{G})=\{(x_k)\in w(G): \Delta_{G}x_k\in c_{0}^{G}\}$
\end{enumerate} 
 are Banach spaces with respect to the norm $||.||^{G}_{\Delta_G}.$ Also these spaces are BK-space.
\end{rem}
Now we define $s: l^G_\infty(\Delta_G)\rightarrow l^G_\infty(\Delta_G), x\rightarrow sx=y=(1, x_2, x_3,....).$ It is clear that $s$ is a bounded linear operator on $l^G_\infty(\Delta_G)$ and $||s||^G_\infty =e.$ Also 
\[s\left[l^G_\infty(\Delta_G)\right] = sl^G_\infty(\Delta_G)=\{x=(x_k): x\in l^G_\infty(\Delta_G), x_1=1 \}\subset l^G_\infty(\Delta_G)\]
is a subspace of $l^G_\infty(\Delta_G)$ and as $|x_1|^G=1$ for $x_1=1$ we have
\[||x||^G_{\Delta_G}= ||\Delta_G x||^G_\infty \quad \text{in}\, sl^G_\infty(\Delta_G).\]
On the other hand we can show that
\begin{equation}\label{eqna}
\Delta_G :sl^G_\infty(\Delta_G)\rightarrow l^G_\infty
\end{equation}
\[x=(x_k)\rightarrow y=(y_k)=(x_k\ominus x_{k+1})\]
is a linear homomorphism. So $sl^G_\infty(\Delta_G)$ and $l^G_\infty$ are equivalent as topological space. $\Delta_G$ and $\Delta_G^{-1}$ are norm preserving and $||\Delta_G||^G_\infty=||\Delta_G^{-1}||^G_\infty =e.$

Let $\left[sl^G_\infty(\Delta_G)\right]^*$ and $\left[l^G_\infty\right]^*$ denote the continuous duals of $sl^G_\infty(\Delta_G)$ and $l^G_\infty,$ respectively.\\
We can prove that
\begin{equation*}
T:\left[sl^G_\infty(\Delta_G)\right]^*\rightarrow \left[l^G_\infty\right]^*,\,
f_{\Delta_G}\rightarrow f= f_{\Delta_G}o\Delta_G^{-1}
\end{equation*}
is a linear isometry. Thus $\left[sl^G_\infty(\Delta_G)\right]^*$ is equivalent to $\left[l^G_\infty\right]^*.$ In the same way we can show that $sc^{G}(\Delta_G)$ and $c^{G},$ $sc_0^{G}(\Delta_G)$ and $c_0^{G}$ are equivalent as topological spaces and \[\left[sc^{G}\Delta_G)\right]^*=\left[sc_0^{G}(\Delta_G)\right]^*=l_1^G \,(l_1^G, \,\text{the space of geometric absolutely convergent series}).\]
\section{Dual spaces of $l^G_\infty (\Delta_G)$}
\begin{lemma}\label{1}
The following conditions (a) and (b) are equivalent:
\begin{align*}
&(a) \sup_k|x_k\ominus x_{k+1}|^G<\infty ~~ i.e. ~~ \sup_k|\Delta_G x_k|^G<\infty;\\
&(b)(i) \sup_k e^{k^{-1}}\odot|x_k|^G<\infty \text{~and}\\
& \quad (ii)\sup_k|x_k\ominus e^{{k(k+1)}^{-1}}\odot x_{k+1}|^G<\infty. 
\end{align*}
\end{lemma}
\begin{proof}
Let (a) be true i.e. $\sup_k|x_k\ominus x_{k+1}|^G<\infty .$
\begin{align*}
\text{Now~} |x_1\ominus x_{k+1}|^G&=\left|{_G\sum^k_{v=1}}{(x_v \ominus x_{v+1})}\right|^G\\
 &=\left|{_G\sum^k_{v=1}}{\Delta_Gx_v}\right|^G\\
&\leq {_G\sum^k_{v=1}}\left|\Delta_Gx_v\right|^G=O(e^k)\\
	\text{and~} |x_k|^G  &=|x_1\ominus x_1\oplus x_{k+1}\oplus x_k\ominus x_{k+1}|^G\\
	 &\leq |x_1|^G\oplus |x_1\ominus x_{k+1}|^G\oplus |x_k\ominus x_{k+1}|^G=O(e^k).			
	 \end{align*}
This implies that $\sup_k e^{k^{-1}}\odot|x_k|^G<\infty.$ This completes the proof of $b(i).$\\
Again
\begin{align*}
\sup_k\left|x_k\ominus e^{{k(k+1)}^{-1}}\odot x_{k+1}\right|^G &=\left|\left\{e^{{(k+1)}}\odot e^{{(k+1)}^{-1}}\right\}\odot x_k\ominus e^{{k(k+1)}^{-1}}\odot x_{k+1} \right|^G \\
  &=\left|\left\{(e^k \oplus e)\odot e^{{(k+1)}^{-1}}\right\}\odot x_k\ominus e^{{k(k+1)}^{-1}}\odot x_{k+1} \right|^G \\
  &=\left|\left\{e^{k(k+1)^{-1}}\odot x_k\oplus e^{(k +1)^{-1}}\odot x_k \right\}\ominus e^{k(k+1)^{-1}}\odot x_{k+1}\right|^G\\
	&=\left|\left\{e^{k(k+1)^{-1}}\odot(x_k\ominus x_{k+1})\right\}\oplus \left\{e^{(k+1)^{-1}}\odot x_k\right\}\right|^G\\
	&\leq e^{k(k+1)^{-1}}\odot \left|x_k\ominus x_{k+1}\right|^G\oplus e^{(k+1)^{-1}}\odot \left|x_k\right|^G\\
			&=O(e).
\end{align*}
Therefore $\sup_k|x_k\ominus e^{{k(k+1)}^{-1}}\odot x_{k+1}|^G<\infty.$ This completes the proof of $b(ii).$

Conversely let $(b)$ be true. Then
\begin{align*}
\left|x_k\ominus e^{k(k+1)^{-1}}\odot x_{k+1}\right|^G&=\left|e^{(k+1)(k+1)^{-1}}\odot x_k\ominus e^{k(k+1)^{-1}}\odot x_{k+1}\right|^G\\
                                                         &\geq e^{k(k+1)^{-1}}\odot|x_k\ominus x_{k+1}|^G\ominus e^{(k+1)^{-1}}\odot |x_k|^G
\end{align*}
i.e. $e^{k(k+1)^{-1}}\odot|x_k\ominus x_{k+1}|^G\leq e^{(k+1)^{-1}}\odot |x_k|^G\oplus \left|x_k\ominus e^{k(k+1)^{-1}}\odot x_{k+1}\right|^G.$\\
Thus $\sup_k|x_k\ominus x_{k+1}|^G<\infty$ as $b(i)$ and $b(ii)$ hold.
\end{proof}
\textbf{Geometric form of Abel's partial summation formula:}
 Abel's partial summation formula states that if $(a_k)$ and $(b_k)$ are sequences, then 
\[\sum_{k=1}^n a_kb_k=\sum_{k=1}^nS_k(b_k-b_{k+1})+ S_nb_{n+1},\]
where $S_k=\sum_{i=1}^ka_i.$ Then
\begin{align*}
\sum_{k=1}^\infty a_kb_k &=\sum_{k=1}^\infty S_k(b_k-b_{k+1})+ \lim_{n\to \infty}S_nb_{n+1}\\
\sum_{k=1}^\infty a_kb_k &=\sum_{k=1}^\infty S_k(b_k-b_{k+1}), \text{~if $(b_k)$ ~ monotonically decreases to zero.}
\end{align*}
Similarly as $\odot$ is distributive over $\oplus$ we have
\[{_G\sum_{k=1}^\infty} a_k\odot b_k ={_G\sum_{k=1}^\infty} S_k\odot(b_k\ominus b_{k+1}),\text{~where~} \,S_k={_G\sum_{i=1}^k}a_i.\]
In particular, if $(b_k)=(e^{-k}),$ then $(b_k)$ monotonically decreases to zero. Then
\begin{align*}
_G\sum^\infty_{k=1}a_k\odot e^{-k} &= _G\sum^\infty_{k=1}S_k\odot \left(e^{-k}\ominus e^{-(k+1)}\right)\\
                                   &= _G\sum^\infty_{k=1}S_k\odot e= _G\sum^\infty_{k=1}S_k.
\end{align*}
 
Let $(p_n)$ be a sequence of geometric positive numbers monotonically increasing to infinity. Then $(\frac {e}{p_n}G)$ is a sequence monotonically decreasing to zero(i.e. to $1$).
\begin{lemma}\label{2}
\[\text{If} ~ \sup_n\left|{_G\sum^n_{v=1}}c_v\right|^G\leq \infty \text {~then~} \sup_n\left(p_n\odot\left|{_G\sum^\infty_{k=1}}\frac{c_{n+k-1}}{p_{n+k}}G\right|^G\right)<\infty.\]
\end{lemma}
\begin{proof}
Using this Abel's partial summation formula to $(c_v)$ and $\left(\frac{e}{p_n}G\right)$ we get
\begin{equation}\label{Eqn2}
{_G\sum_{k=1}^\infty} \frac{c_{n+k-1}}{p_{n+k}}G= {_G\sum_{k=1}^\infty}\left({_G\sum_{v=1}^k} c_{n+v-1}\right)\odot \left(\frac{e}{p_{n+k}}G\ominus \frac{e}{p_{n+k+1}}G\right)
\end{equation}
\[\text{and}\quad p_n \odot\left|{_G\sum_{k=1}^\infty}\frac{c_{n+k-1}}{p_{n+k}}G\right|^G= O(e).\]
\end{proof}
\begin{lemma}\label{3} If the series $\sum_{k=1}^\infty c_k$ is convergent then 
\[ \lim_n \left( p_n \odot {_G\sum_{k=1}^\infty}\frac{c_{n+k-1}}{p_{n+k}}G\right)=1.\]
\end{lemma}
\begin{proof}
Since $$ \left|{_G\sum_{v=1}^k} c_{n+v-1}\right|^G = \left| {_G\sum_{v=n}^{n+k-1}} c_v\right|^G=O(e)$$ for every $k\in \mathbb{N}.$ Using (\ref{Eqn2}) we get
\[p_n\odot \left|{_G\sum_{k=1}^\infty} \frac{c_{n+k-1}}{p_{n+k}}G\right|^G=O(e).\]
\end{proof}
\begin{corollary}\label{Cor1} Let $(p_n)$ be monotonically increasing. If
$$\sup_n\left|{_G\sum_{v=1}^n} p_v\odot a_v\right|^G<\infty \text{~then~} \sup_n\left|p_n\odot {_G\sum_{k=n+1}^\infty} a_k\right|^G<\infty.$$
\end{corollary}
\begin{proof}
We put $p_{k+1}\odot a_{k+1}$ instead of $c_k$ in Lemma \ref{2} we get
\begin{align*}
p_n\odot {_G\sum_{k=1}^\infty} \frac{c_{n+k-1}}{p_{n+k}}G &=p_n\odot {_G\sum_{k=1}^\infty} \frac{p_{n+k}\odot a_{n+k}}{p_{n+k}}G\\
  &=p_n\odot {_G\sum_{k=1}^\infty} a_{n+k}\\
																												&=p_n\odot {_G\sum_{k=n+1}^\infty} a_k =O(e).
\end{align*}
\end{proof}
\begin{corollary}\label{Cor2}
$$\text{If}~ {_G\sum_{k=1}^\infty} p_k\odot a_k \text{~is convergent~ then~} \lim_n p_n\odot {_G\sum_{k=n+1}^\infty} a_k=1.$$
\end{corollary}
\begin{proof}
We put $p_{k+1}\odot a_{k+1}$ instead of $c_k$ in Lemma \ref{3}.
\end{proof}
\begin{corollary}\label{Cor3}
 $${_G\sum_{k=1}^\infty} e^k\odot a_k \text{~is convergent iff~}{_G\sum_{k=1}^\infty} R_k\text{~is convergent with~} e^n\odot R_n = O(e),\text{~where~}$$
 $$R_n = {_G\sum_{k=n+1}^\infty} a_k.$$
\end{corollary}
\begin{proof}
Let $p_n=e^n.$ Then it is monotonically increasing to infinity. Then
\begin{align*}
_G\sum_{k=1}^n e^k\odot a_{k+1} &= e\odot a_2 \oplus e^2 \odot a_3 \oplus e^3 \odot a_4 \oplus.....\oplus e^n\odot a_{n+1}\\
                                &= (a_2\oplus a_3 \oplus ....\oplus a_{n+1})\oplus (a_3\oplus a_4\oplus ...\oplus a_{n+1})\\
																&\qquad \oplus ..........\oplus (a_n\oplus a_{n+1})\oplus (a_{n+1})\\
																&= (R_1 \ominus R_{n+1})\oplus (R_2 \ominus R_{n+1})\oplus......\oplus (R_{n-1} \ominus R_{n+1})\oplus (R_n \ominus R_{n+1}) \\
																&= _G\sum_{k=1}^n R_k \ominus \{e^n\odot R_{n+1}\}.
\end{align*}
Therefore as $e^n\odot R_n=O(e)$, so $e^n\odot R_{n+1}=O(e).$ This implies \[{_G\sum_{k=1}^n} e^k\odot a_{k+1}\text{~is convergent if}\, _G\sum_{k=1}^nR_k\text{~is convergent and vice versa.}\]
\end{proof}
\section{$\alpha-,\beta-,$  $\gamma-$ duals}
\begin{defn} \cite{Garling67, KotheToplitz69, KotheToplitz34, Maddox80} If $X$ is a sequence space, we define
\begin{enumerate}
\item[(i)] $X^\alpha=\{a=(a_k) : \sum_{k=1}^\infty |a_k x_k|<\infty, \, \text{for each} \,x\in X\};$
\item[(ii)] $X^\beta=\{a=(a_k) : \sum_{k=1}^\infty a_k x_k \, \text{is convergent, for each} \,x\in X\};$
\item[(iii)] $ X^\gamma=\{a=(a_k) : \sup_n|\sum_{k=1}^n a_k x_k|<\infty, \, \text{for each} \,x\in X\}.$
\end{enumerate}
\end{defn}
$X^\alpha, X^\beta,$ and $X^\gamma$ are called $\alpha-$ (or K\"{o}the-Toeplitz), $\beta-$(or generalised K\"{o}the-Toeplitz), and $\gamma-$dual spaces of $X$. We can show that $X^\alpha \subset X^\beta \subset X^\gamma.$ If $X\subset Y,$ then $Y^\dag\subset X^\dag, $ for $\dag=\alpha, \beta $ or $\gamma.$
\begin{thm}
\[(i)~\text{If~} D_1= \left\{a=(a_k): {_G\sum_{k=1}^\infty} e^k\odot |a_k|^G<\infty\right\} \text{~then~} \left(sl_\infty^G(\Delta_G)\right)^\alpha=D_1.\]
\[(ii)~\text{If~} D_2= \left\{a=(a_k): {_G\sum_{k=1}^\infty} e^k\odot a_k \text{~is convergent with~} {_G\sum_{k=1}^\infty}|R_k|^G<\infty\right\}.\]
Then $\left(sl_\infty^G(\Delta_G)\right)^\beta=D_2.$
\[(iii)~\text{If~} D_3= \left\{a=(a_k): \sup_n|{_G\sum_{k=1}^n} e^k\odot a_k|^G<\infty, {_G\sum_{k=1}^\infty}|R_k|^G<\infty\right\}.\]
Then $\left(sl_\infty^G(\Delta_G)\right)^\gamma=D_3.$
\end{thm} 
\begin{proof}$(i)$ Let $a\in D_1.$ Then for each $x\in sl_\infty^G(\Delta_G)$ we have
\[{_G\sum_{k=1}^\infty} |a_k\odot x_k|^G=_G\sum_{k=1}^\infty \left(e^k\odot |a_k|^G\right)\odot \left(e^{k^{-1}}\odot |x_k|^G\right)<\infty \quad \text{by using Lemma \ref{1}}.\]
This implies that $a\in \left(sl_\infty^G(\Delta_G)\right)^\alpha.$ Therefore
\begin{equation}\label{8}
 D_1\subseteq \left(sl_\infty^G(\Delta_G)\right)^\alpha.
\end{equation}
Again let $a\in \left(sl_\infty^G(\Delta_G)\right)^\alpha.$ Then ${_G\sum_{k=1}^\infty}|a_k\odot x_k|^G<\infty$ (by definition of $\alpha$-dual) for each $x\in sl_\infty^G(\Delta_G).$ So we take 
\begin{equation*}
x_k=
\begin{cases}
1, &\text{if $k=1;$}\\
e^k,&\text{if $k\geq 2,$}
\end{cases}
\end{equation*}
then $x=(1, e^2, e^3,.....)\in sl_\infty^G(\Delta_G).$ Therefore 
\begin{align*}
{_G\sum_{k=1}^\infty} e^k\odot|a_k|^G &=|a_1|^G \oplus {_G\sum_{k=2}^\infty} e^k\odot|a_k|^G\\
 &= |a_1|^G \oplus {_G\sum_{k=1}^\infty} |a_k\odot x_k|^G<\infty~ \text{as}~ a_1\odot x_1=1.
\end{align*}
Therefore $a\in D_1.$ This implies that
\begin{equation}\label{9}
\left(sl_\infty^G(\Delta_G)\right)^\alpha \subseteq D_1.
\end{equation}
Therefore from (\ref{8}) and (\ref{9}) we get
\[\left(sl_\infty^G(\Delta_G)\right)^\alpha = D_1.\]
$(ii)$ Let $a\in D_2.$ If $x\in sl_\infty^G(\Delta_G)$ then there exists one and only one $y=(y_k)\in l_\infty^G$ such that (see \ref{eqna})
\begin{align*}
    x_k &= \ominus {_G\sum_{v=1}^k} y_{v-1},  \,y_0=1\\
	 \text{Therefore}\quad x_1 &= \ominus {_G\sum_{v=1}^1} y_{v-1}	 =\ominus y_o=1\\
	  x_2 &= \ominus {_G\sum_{v=1}^2} y_{v-1}	 =\ominus y_1\\
		x_3 &= \ominus {_G\sum_{v=1}^3} y_{v-1}	 =\ominus y_1\ominus y_2\\
		x_4 &= \ominus {_G\sum_{v=1}^4} y_{v-1}	 =\ominus y_1\ominus y_2\ominus y_3\\
		...&.................................................\\
		...&................................................. \quad .\\
		  \text{Then~} {_G\sum_{k=1}^n} a_k\odot x_k &=a_1\odot x_1\oplus a_2\odot x_2\oplus a_3\odot x_3\oplus ......\oplus a_n\odot x_n\\
			 &=a_1 \odot 1\ominus a_2\odot y_1\ominus a_3\odot(y_1\oplus y_2)\ominus a_4\odot (y_1\oplus y_2\oplus y3)\ominus\\
			 &\qquad .....\ominus a_n\odot(y_1\oplus y_2\oplus .....\oplus y_{n-1})\\
			 &= \ominus(a_2\oplus a_3\oplus....\oplus a_n)\odot y_1\\
			 &\qquad \qquad \ominus (a_3\oplus a_4\oplus.... \oplus a_n)\odot y_2\ominus ....... \ominus a_n\odot y_{n-1}\\
			 &=(\ominus R_1 \odot y_1\oplus R_n\odot y_1)\oplus (\ominus R_2 \odot y_2\oplus R_n\odot y_2)\oplus ....\\
			 & \qquad \qquad.........\oplus (\ominus R_{n-1} \odot y_{n-1}\oplus R_n\odot y_{n-1})\\
			 &= \ominus _G\sum_{k=1}^{n-1}R_k\odot y_k \oplus R_n\odot _G\sum_{k=1}^{n-1}y_k.
\end{align*}
\begin{equation}\label{Eqn3}
 {_G\sum_{k=1}^n} a_k\odot x_k =\ominus {_G\sum_{k=1}^{n-1}}R_k\odot y_k \oplus R_n\odot{ _G\sum_{k=1}^{n-1}}y_k. 
\end{equation}
\[\text{Since~} {_G\sum_{k=1}^\infty} R_k\odot y_k \text{~is absolutely convergent and~} R_n\odot {_G\sum_{k=1}^{n-1}}y_k \rightarrow 1 \text{~as~} n\rightarrow \infty (\text{~Corollary \ref{Cor3}}),\]
 \[\text{the series~} {_G\sum_{k=1}^n} a_k\odot x_k \,\text{is convergent for each}\, x\in  sl_\infty^G(\Delta_G). \text{~This yields~} a\in \left( sl_\infty^G(\Delta_G)\right)^\beta. \]
 Therefore
$D_2 \subseteq \left( sl_\infty^G(\Delta_G)\right)^\beta.$\\
\[\text{Again let~}  a\in \left( sl_\infty^G(\Delta_G)\right)^\beta  \text{~then~} {_G\sum_{k=1}^\infty} a_k\odot x_k \text{~is convergent for each~} x\in  sl_\infty^G(\Delta_G). \text{~We take}\]
\begin{equation*}
x_k=
\begin{cases}
1, &\text{if $k=1$;}\\
e^k,&\text{if $k\geq 2.$}
\end{cases}
\end{equation*}
\[\text{Thus~} {_G\sum_{k=1}^\infty} e^k\odot x_k \, \text{is convergent. This implies~} e^n\odot R_n=O(e) (\text{~Corollary \ref{Cor3}}).\]
\[\text{Using (\ref{Eqn3}) we get~~} {_G\sum_{k=1}^\infty} a_k\odot x_k= \ominus {_G\sum_{k=1}^\infty} R_k\odot y_k \text{~converges for all~} y\in l_\infty^G. \text{~So we have}\] 
\[{_G\sum_{k=1}^\infty}|R_k|^G <\infty\, \text{~and~} a\in D_2. \]
Therefore
\[\left( sl_\infty^G(\Delta_G)\right)^\beta=D_2.\]
$(iii)$ The proof of this part is same as above.
\end{proof}
\section{Some applications of Geometric Difference}
In this section we find the Geometric Newton-Gregory interpolation formulae and solve some numerical problems using these new formulae. 
\begin{description}
\item[Geometric Factorial]Let us define geometric factorial notation $!_G$ as
\[n!_G=e^n\odot e^{n-1}\odot e^{n-2}\odot \cdots \odot e^2\odot e =e^{n!}.\]
For example,
\begin{align*}
0!_G &=e^{0!}=e^0=1\\
1!_G &=e^{1!}=e=2.71828\\
2!_G &=e^{2!}=e^2=7.38906\\
3!_G &=e^{3!}=e^6=4.03429 \times 10^2\\
4!_G &=e^{4!}=e^{24}=2.64891 \times 10^{10}\\
5!_G &=e^{5!}=e^{120}=1.30418 \times 10^{52}\quad \text{etc.}
\end{align*}
\item[Generalized Geometric Forward Difference Operator]Let
\begin{align*}
\Delta_G f(a)   &= f(a\oplus h) \ominus f(a).\\
\Delta^2_G f(a) &= \Delta_G f(a\oplus h) \ominus \Delta_G f(a)\\
                &= \{f(a\oplus e^2\odot h) \ominus f(a\oplus h)\}\ominus \{f(a\oplus h) \ominus f(a)\}\\                
								&=f(a\oplus e^2\odot h) \ominus e^2 \odot f(a\oplus h)\oplus f(a).\\
\Delta^3_G f(a) &= \Delta^2_G f(a\oplus h) \ominus \Delta^2_G f(a)\\
                &=\{f(a\oplus e^3\odot h) \ominus e^2 \odot f(a\oplus e^2\odot h)\oplus f(a \oplus h)\}\\
								&\qquad \ominus \{f(a\oplus e^2\odot h) \ominus e^2 \odot f(a\oplus h)\oplus f(a)\}\\
								&=f(a\oplus e^3\odot h) \ominus e^3 \odot f(a\oplus e^2\odot h)\oplus e^3\odot f(a \oplus h) \ominus f(a).
\end{align*}
Thus, $n^{\text{th}}$ geometric forward difference is
\[\Delta^n_G f(a)= _G\sum^n_{k=0} (\ominus e )^{{k}_G}\odot e^{\binom{n}{k}}\odot f(a\oplus e^{n-k}\odot h), \text{with}\, (\ominus e)^{0_G}=e.\]
\item[Generalized Geometric Backward Difference Operator] Let
\begin{align*}
\nabla_G f(a)   &=f(a) \ominus f(a\ominus h).\\
\nabla^2_G f(a) &= \nabla_G f(a) \ominus \nabla_G f(a\ominus h)\\
                &= \{f(a)\ominus f(a \ominus h)\} \ominus \{f(a\ominus h) \ominus f(a\ominus e^2\odot h)\}\\
								&= f(a)\ominus  e^2 \odot f(a \ominus h)\oplus f(a \ominus e^2 \odot h).\\
\nabla^3_G f(a) &= \nabla^2_G f(a) \ominus \nabla^2_G f(a-h)\\
                &=\{f(a)\ominus  e^2 \odot f(a \ominus h)\oplus f(a \ominus e^2 \odot h)\}\\
								&\qquad \ominus \{f(a\ominus h)\ominus  e^2 \odot f(a \ominus e^2 \odot h)\oplus f(a \ominus e^3 \odot h)\}\\
								&=f(a) \ominus e^3 \odot f(a \ominus h)\oplus e^3\odot f(a \ominus e^2 \odot h) \ominus f(a\ominus e^3 \odot h).
\end{align*}
Thus, $n^{\text{th}}$ geometric backward difference is
\[\nabla^n_G f(a)= _G\sum^n_{k=0} (\ominus e )^{{k}_G}\odot e^{\binom{n}{k}}\odot f(a\ominus e^k\odot h).\]
\item[Factorial Function]\textit{The product of n consecutive factors each at a constant\\
 geometric difference, h, the first factor being x is called a factorial function of degree n and is denoted by $x^{(n_G)}.$}Thus
\[x^{(n_G)}=x\odot (x \ominus e\odot h)\odot(x \ominus e^2\odot h)\odot(x \ominus e^3\odot h)\odot\cdots \odot (x \ominus e^{n-1}\odot h).\]
In particular, for $h=e,$
\[x^{(n_G)}=x\odot (x \ominus e)\odot (x \ominus e^2)\odot (x \ominus e^3)\odot\cdots \odot (x \ominus e^{n-1}).\]
\end{description}
\textbf{Geometric Newton-Gregory Forward Interpolation Formula:}
Let $y=f(x)$ be a function which takes the values $f(a),f(a\oplus h), f(a\oplus e^2\odot h), f(a\oplus e^3\odot h),......,f(a\oplus e^n\odot h)$ for the $n+1$ geometrically equidistant values (which form a Geometric Progression in ordinary sense) $a, a\oplus h, a\oplus e^2\odot h, a\oplus e^3\odot h,......, a\oplus e^n\odot h$ of the independent variable $x$ and let $P_n(x)$ be a geometric polynomial in $x$ of degree $n$ defined as:
\begin{gather}
\begin{aligned}\label{eqn10}
P_n(x)=& A_0\oplus A_1 \odot (x\ominus a)\oplus A_2\odot (x\ominus a)\odot(x\ominus a\ominus h)\\
      &\oplus A_3\odot(x\ominus a)\odot(x\ominus a \ominus h)\odot(x\ominus a\ominus e^2\odot h)\oplus\cdots\\
			&\oplus A_n \odot (x\ominus a)\odot(x\ominus a\ominus h)\odot\cdots \odot(x\ominus a\ominus e^{n-1}\odot h).
\end{aligned}
\end{gather}
We choose the coefficients $A_0, A_1, A_2,....,A_n$ such that\\
 $P_n(a)=f(a),P_n(a\oplus h)=f(a\oplus h), P_n(a \oplus e^2\odot h)=f(a \oplus e^2\odot h),.... ,P_n(a \oplus e^n\odot h)=f(a \oplus e^n\odot h).$

Putting $x= a, a\oplus h, a\oplus e^2\odot h, a\oplus e^3\odot h,......, a\oplus e^n\odot h$ in (\ref{eqn10}) and then also putting the values of $P_n(a), P_n(a\oplus h),......., P_n(a\oplus e^n\odot h),$ we get
\[f(a)=A_0\implies A_0=f(a).\]
\[f(a\oplus h)=A_0\oplus A_1\odot h \implies A_1=\frac{f(a\oplus h) \ominus f(a)}{h}G=\frac{\Delta_G f(a)}{h}G.\]
\begin{align*}
f(a\oplus e^2\odot h)&=A_0\oplus e^2\odot h\odot A_1 \oplus e^2\odot h\odot h\odot A_2\\
\implies A_2         &=\frac{f(a\oplus e^2\odot h) \ominus e^2\odot [f(a\oplus h)\ominus f(a)]\ominus f(a)}{e^2\odot h^{2_G}}G\\
                     &= \frac{f(a\oplus e^2\odot h) \ominus e^2\odot f(a\oplus h)\oplus f(a)}{2!_G\odot h^{2_G}}G\\
										 &=\frac{\Delta^2_G f(a)}{2!_G\odot h^{2_G}}G.\\
\text{Similarly}\quad A_3 &=\frac{\Delta^3_G f(a)}{3!_G\odot h^{3_G}}G\\
             \cdots  &\quad \cdots \quad \cdots \quad \cdots\\
						    A_n  &=\frac{\Delta^n_G f(a)}{n!_G\odot h^{n_G}}G.
\end{align*}
Putting the values of $A_0, A_1, A_2,....,A_n$ found above in (\ref{eqn10}), we get
\begin{align*}
P_n(x)=& f(a)\oplus \frac{\Delta_G f(a)}{h}G \odot (x\ominus a)\oplus \frac{\Delta^2_G f(a)}{2!_G\odot h^{2_G}}G\odot (x\ominus a)\odot(x\ominus a\ominus h)\\
      &\oplus \frac{\Delta^3_G f(a)}{3!_G\odot h^{3_G}}G\odot(x\ominus a)\odot(x\ominus a \ominus h)\odot(x\ominus a\ominus e^2\odot h)\oplus\cdots\\
			&\oplus \frac{\Delta^n_G f(a)}{n!_G\odot h^{n_G}}G \odot (x\ominus a)\odot(x\ominus a\ominus h)\odot\cdots \odot(x\ominus a\ominus e^{n-1}\odot h).
\end{align*}
This is the Geometric Newton-Gregory forward interpolation formula.

Putting ${\frac{x\ominus a}{h}}G= u$ or $x=a \oplus h\odot u,$ formula takes the form
\begin{gather}
\begin{aligned}\label{eqn11}
P_n(x)=& f(a)\oplus u\odot \Delta_G f(a) \oplus \frac{u\odot(u\ominus e)}{2!_G}G\odot \Delta^2_G f(a)\\
      &\oplus \frac{u\odot(u\ominus e)\odot (u \ominus e^2)}{3!_G}G \odot \Delta^3_G f(a)\oplus\cdots\\
			&\oplus \frac{u\odot (u \ominus e)\odot(u \ominus e^2)\odot \cdots \odot (u \ominus e^{n-1})}{n!_G}G \odot \Delta^n_G f(a).
\end{aligned}
\end{gather}
The result (\ref{eqn11}) can be written as
\begin{align*}
P_n(x)=P_n(a\oplus h\odot u)=&f(a)\oplus u^{(1_G)}\odot \Delta_G f(a) \oplus \frac{u^{(2_G)}}{2!_G}G\odot \Delta^2_G f(a)\oplus \frac{u^{(3_G)}}{3!_G}G \odot \Delta^3_G f(a)\oplus \cdots \\
			                       & \cdots \oplus \frac{u^{(n_G)}}{n!_G}G \odot \Delta^n_G f(a).
\end{align*}
where $u^{(n_G)}=u\odot (u \ominus e)\odot (u \ominus e^2)\odot\cdots \odot (u \ominus e^{n-1}).$
\begin{example} Given,$f(x)=f(e^t)=\sin(e^t).$ From the following table, find $\sin(e^{1.3})$ using geometric forward interpolation formula.\\[2ex]
\begin{center}
\begin{tabular}{|c| c| c| c|c|} 
\hline
$x$  & $e$ &$e^{1.2}$  & $e^{1.4}$& $e^{1.6}$\\ [0.5ex]
\hline 
$f(x)$&$0.0474$& $0.0579$&$0.0707$&$0.0863$\\[1ex] 
\hline 
\end{tabular}
\end{center}
\noindent \textbf{Solution.}
The geometric difference table for given data is as follows:\\[1.5ex]
\begin{center}
\begin{tabular}{|c| c| c| c| c|} 
\hline
$x$      & $f(x)$ & $\Delta_G f(x)$  & $\Delta^2_G f(x)$& $\Delta^3_G f(x)$\\ [1.5ex]
\hline 
$e$      & 0.0474  &                  &                  & \\
         &         &1.2215            &                  & \\
$e^{1.2}$& 0.0579  &                  & 0.9997           & \\
         &         &1.2211            &                  & 0.9999 \\
$e^{1.4}$& 0.0707  &                  & 0.9996           & \\
         &         &1.3306            &                  & \\
$e^{1.6}$& 0.0863  &                  &                  & \\
\hline 
\end{tabular}
\end{center}
We have to calculate
\begin{align*} 
f(e^{1.3})&=f(a\oplus u\odot h),\, \text{say}.\\
\therefore \quad  a\oplus u\odot h&= e^{1.3}\\
\Rightarrow e\oplus u \odot e^{0.2}&= e^{1.3}, \quad(\text{here} ~ h=e^{1.2}\ominus e=e^{0.2})\\
u&= \frac{e^{1.3} \ominus e}{e^{0.2}}G\\
&=\left(e^{0.3}\right)^\frac{1}{0.2}\\
& =e^{1.5}
\end{align*}
By Geometric Newton-Gregory forward interpolation formula we get
\begin{align*}
f(a\oplus u\odot h)&=f(a)\oplus u\odot \Delta_G f(a) \oplus \frac{u\odot (u \ominus e)}{e^2}G\odot \Delta^2_G f(a)\\
       &\quad\oplus \frac{u\odot (u \ominus e)\odot (u \ominus e^2)}{e^6}G \odot \Delta^3_G f(a)\\
f(e^{1.3}) &= f(e)\oplus \{e^{1.5}\odot \Delta_G f(e)\} \oplus \{\frac{e^{1.5}\odot (e^{1.5}\ominus e)}{e^2}G\odot \Delta^2_G f(e)\}\\
      & \quad \oplus \{\frac{e^{1.5}\odot (e^{1.5}\ominus e)\odot (e^{1.5}\ominus e^2)}{e^6}G \odot \Delta^3_G f(e)\}\\
			&= 0.0474\oplus \{e^{1.5}\odot 1.2215\}\oplus \{\frac{e^{1.5}\odot e^{0.5}}{e^2}G\odot 0.9997\}\\
      & \quad \oplus \{\frac{e^{1.5}\odot e^{0.5}\odot e^{-0.5}}{e^6}G \odot 0.9999\}\\
		  &=0.0474\oplus (1.2215)^{1.5} \oplus (0.9997)^{0.325} \oplus (0.9999)^{\frac{1}{0.0625}}\\
			&=0.0474\oplus 1.3500 \oplus 0.9999 \oplus 0.9984\\
			&=0.0474 \times 1.3500 \times 0.9999 \times 0.9984\\
			&=0.0639
\end{align*}
Thus $\sin(e^{1.3})=0.0639.$
\end{example}
\textbf{Note:} It is to be noted that $e^x\odot e^y=e^{xy},e^x\oplus e^y=e^{x+y}, x\oslash e^y=x^{\frac{1}{y}}.$
\vspace{0.3cm}\\
\textbf{Geometric Newton-Gregory Backward Interpolation Formula:}
Let $y=f(x)$ be a function which takes the values $f(a\oplus e^n\odot h),f(a\oplus e^{n-1}\odot h), f(a\oplus e^{n-2}\odot h), f(a\oplus e^{n-3}\odot h),......,f(a)$ for the $n+1$ geometrically equidistant values $a\oplus e^n\odot h, a\oplus e^{n-1}\odot h, a\oplus e^{n-2}\odot h, a\oplus e^{n-3}\odot h,......, a$ of the independent variable $x$ and let $P_n(x)$ be a geometric polynomial in $x$ of degree $n$ defined as:
\begin{gather}
\begin{aligned}\label{eqn12}
P_n(x)=& A_0\oplus A_1 \odot (x\ominus a\ominus e^n\odot h)\oplus A_2\odot (x\ominus a \ominus e^n\odot h)\odot(x\ominus a \ominus e^{n-1}\odot h)\\
      &\oplus A_3\odot(x\ominus a \ominus e^n\odot h)\odot(x\ominus a \ominus e^{n-1} \odot h)\odot(x\ominus a\ominus e^{n-2}\odot h)\oplus\cdots\\
			&\oplus A_n \odot (x\ominus a \ominus e^n\odot h)\odot(x\ominus a \ominus e^{n-1} \odot h)\odot\cdots \odot(x\ominus a\ominus h).
\end{aligned}
\end{gather}
where $A_0, A_1, A_2,......,A_n$ are constants which are to be determined so as to make
\[P_n(a\oplus e^n\odot h)=f(a\oplus e^n\odot h), P_n(a\oplus e^{n-1}\odot h)=f(a\oplus e^{n-1}\odot h),..., P_n(a)=f(a)\]
Putting $x=a\oplus e^n\odot h, a\oplus e^{n-}\odot h,$.... in (\ref{eqn12}) and also putting\\
 $P_n(a\oplus e^n\odot h)=f(a\oplus e^n\odot h),$....., we get
\begin{align*}
A_0&=f(a\oplus e^n\odot h)\\
A_1&=\frac{\nabla_G f(a\oplus e^n\odot h)}{h}G\\
A_2&=\frac{\nabla^2_G f(a\oplus e^n\odot h)}{2!_G\odot h^{2_G}}G\\
A_3&=\frac{\nabla^3_G f(a\oplus e^n\odot h)}{3!_G\odot h^{3_G}}G\\
....&\quad .............................\\
A_n&=\frac{\nabla^n_G f(a\oplus e^n\odot h)}{n!_G\odot h^{n_G}}G
\end{align*}
Substituting the values of $A_0, A_1, A_2,....$ in (\ref{eqn12}), we get
\begin{gather}
\begin{aligned}\label{eqn13}
P_n(x)=& f(a\oplus e^n\odot h)\oplus \frac{\nabla_G f(a\oplus e^n\odot h)}{h}G \odot (x\ominus a\ominus e^n\odot h)\\
&\oplus \frac{\nabla^2_G f(a\oplus e^n\odot h)}{2!_G\odot h^{2_G}}G\odot (x\ominus a \ominus e^n\odot h)\odot(x\ominus a \ominus e^{n-1}\odot h)\\
      &\oplus \frac{\nabla^3_G f(a\oplus e^n\odot h)}{3!_G\odot h^{3_G}}G\odot(x\ominus a \ominus e^n\odot h)\odot(x\ominus a \ominus e^{n-1} \odot h)\odot(x\ominus a\ominus e^{n-2}\odot h)\oplus\cdots\\
			&\oplus \frac{\nabla^n_G f(a\oplus e^n\odot h)}{n!_G\odot h^{n_G}}G \odot (x\ominus a \ominus e^n\odot h)\odot(x\ominus a \ominus e^{n-1} \odot h)\odot\cdots \odot(x\ominus a\ominus h).
\end{aligned}
\end{gather}
This is the Geometric Newton-Gregory backward interpolation formula.

Putting $u=\frac{x\ominus (a\oplus e^n\odot h)}{h}G$ or $x=a\oplus e^n\odot h\oplus u\odot h,$ we get
\begin{gather}
\begin{aligned}\label{eqn14}
P_n(x)=&=P_n(a\oplus e^n\odot h\oplus u\odot h)= f(a\oplus e^n\odot h)\oplus u \odot \nabla_G f(a\oplus e^n\odot h)\\
&\oplus \frac{u\odot(u \oplus e)}{2!_G}G\odot \nabla^2_G f(a\oplus e^n\odot h)\\
      &\oplus \frac{u\odot(u \oplus e)\odot(u \oplus e^2)}{3!_G}G\odot \nabla^3_G f(a\oplus e^n\odot h)\oplus\cdots\\
			&\oplus \frac{u\odot(u \oplus e)\odot(u \oplus e^2)\odot \cdots \odot (u \oplus e^{n-1})}{n!_G}G \odot \nabla^n_G f(a\oplus e^n\odot h).
\end{aligned}
\end{gather}
\begin{example} Given,$f(x)=\ln (x).$ From the following table, find $\ln(22)$ using geometric backward interpolation formula.\\[3ex]
\begin{center}
\begin{tabular}{|c| c| c| c|c|} 
\hline
$x$  & 3 & 6  & 12 & 24\\ [0.5ex]
\hline 
$f(x)$& 1.0986 & 1.7918 & 2.4849 & 3.1781\\[1ex] 
\hline 
\end{tabular}
\end{center}
\textbf{Solution.}
The geometric difference table for given data is as follows:\\[1ex]
\begin{center}
\begin{tabular}{|c| c| c| c| c|} 
\hline
$x$& $f(x)$ & $\nabla_G f(x)$  & $\nabla^2_G f(x)$& $\nabla^3_G f(x)$\\ [1.5ex]
\hline 
3  & 1.0986  &                  &                  & \\
         &         &1.6310            &                  & \\
6  & 1.7918  &                  & 0.8503           & \\
         &         &1.3868            &                  & 1.0847 \\
12 & 2.4849  &                  & 0.9223           & \\
         &         &1.2790            &                  & \\
24 & 3.1781  &                  &                  & \\
\hline 
\end{tabular}
\end{center}
We have to compute
\begin{align*}
f(22)&=f(a\oplus e^n\odot h\oplus u\odot h),\, \text{say}.\\
\therefore \quad  a\oplus e^n\odot h\oplus u\odot h &= 22\\
\Rightarrow 24\oplus u\odot h&= 22, \quad(\text{here} ~ h=6\ominus 3= 2)\\
u&= \frac{22 \ominus 24}{2}G\\
&=\left(0.9167\right)^\frac{1}{\ln 2}\\
& =0.8820.
\end{align*}
By Geometric Newton-Gregory backward interpolation formula we get
\begin{align*}
f(22)&=f(24)\oplus u \odot \nabla_G f(24) \oplus \frac{u\odot(u \oplus e)}{2!_G}G\odot \nabla^2_Gf(24)\\
        & \quad \oplus \frac{u\odot(u \oplus e)\odot(u \oplus e^2)}{3!_G}G\odot \nabla^3_Gf(24)\\
		 &= 3.1781\oplus \{0.8820 \odot 1.2790\}  \oplus \{\frac{0.8820\odot(0.8820 \oplus e)}{e^2}G\odot 0.9223\} \\
      &\quad \oplus \{\frac{0.8820\odot(0.8820 \oplus e)\odot(0.8820 \oplus e^2)}{e^6}G\odot 1.0847\}\\
		 &= 3.1781\oplus 0.9696 \oplus \{0.9466\odot 0.9223\}\oplus \{0.9663 \odot 1.0847\}\\
		 &=3.1781 \oplus 0.9696 \oplus 1.0045 \oplus 0.9972\\
		 &= 3.0867
\end{align*}
Therefore $\ln (22)= 3.0867.$
\end{example}
\textbf{Note:} Since small change in $x$ results large change in $e^x.$ So, for better accuracy, values should be taken up to maximum possible decimal places.


\textbf{Advantages of Geometric Interpolation Formulae over Ordinary Interpolation Formulae:} All the ordinary interpolation formulae are based upon the fundamental assumption that the data is expressible or can be expressed as a polynomial function with fair degree of accuracy. But geometric interpolation formulae have no such restriction. Because geometric interpolation formulae are based on geometric polynomials which are not polynomials in ordinary sense. So geometric interpolation formulae can be used to generate  transcendental functions, mainly to compute exponential and logarithmic functions. Also geometric forward and backward interpolation formulae are based on the values of the argument that are geometrically equidistant but need not be equidistant like classical interpolation formulae.

\section{Conclusion}
In this paper, we have defined geometric difference sequence space and obtained the Geometric Newton-Gregory interpolation formulae. Our main aim is to bring up geometric calculus to the attention of researchers in the branch of numerical analysis and to demonstrate its usefulness. 
We think that geometric calculus may especially be useful as a mathematical tool
for economics, management and finance.
\end{document}